\documentclass[11pt,twoside,english]{amsart}
\advance\oddsidemargin by -1.0cm
\advance\evensidemargin by -1.0cm
\advance\topmargin by -1.0cm 
\usepackage{amssymb}
\usepackage{babel}
\usepackage{amstext}
\usepackage{amscd}   
\usepackage{epsfig}  
\usepackage{rotating}
\let\mathg\mathfrak
\theoremstyle{plain}
\newtheorem{cor}{Corollary}[section]
\newtheorem{lem}{Lemma}[section]
\newtheorem{thm}{Theorem}[section]            
\newtheorem{prop}{Proposition}[section]
\theoremstyle{definition}
\newtheorem{exa}{Example}[section]
\newtheorem{NB}{Remark}[section]

\newcommand{\bdm}{\begin{displaymath}}
\newcommand{\edm}{\end{displaymath}}
\newcommand{\be}{\begin{equation}}
\newcommand{\ee}{\end{equation}}
\newcommand{\ba}[1]{\begin{array}{#1}}
\newcommand{\ea}{\end{array}}

\newcommand{\btab}{\begin{tabular}}
\newcommand{\etab}{\end{tabular}}

\newcommand{\Id}{\ensuremath{\mathrm{Id}}}

\newcommand{\C}{\ensuremath{\mathbb{C}}}
\newcommand{\R}{\ensuremath{\mathbb{R}}}
\newcommand{\CP}{\ensuremath{\mathbb{CP}}}

\newcommand{\T}{\ensuremath{\mathrm{T}}}

\newcommand{\G}{\ensuremath{\mathrm{G}}}

\newcommand{\Ric}{\ensuremath{\mathrm{Ric}}}
\newcommand{\Scal}{\ensuremath{\mathrm{Scal}}}

\newcommand{\Lin}{\ensuremath{\mathrm{Lin}}}

\newcommand{\un}{\ensuremath{\mathg{u}}}
\newcommand{\su}{\ensuremath{\mathg{su}}}
\newcommand{\SU}{\ensuremath{\mathrm{SU}}}

\newcommand{\so}{\ensuremath{\mathg{so}}}
\newcommand{\hol}{\ensuremath{\mathg{hol}}}

\newcommand{\g}{\ensuremath{\mathfrak{g}}}

\begin{document}
\def\haken{\mathbin{\hbox to 6pt{%
                 \vrule height0.4pt width5pt depth0pt
                 \kern-.4pt
                 \vrule height6pt width0.4pt depth0pt\hss}}}
    \let \hook\intprod
\setcounter{equation}{0}
%
%------ draw title page -----
%
\thispagestyle{empty}
%
%\hbox to \hsize{%
%  \vtop{} \hfill
%  \vtop{\hbox{PRELIMINARY VERSION}}}
%------------------------------
\date{\today}
%----------------------------------------------------------
\title[Cocalibrated $\G_2$-manifolds with Ricci flat characteristic connection]
{Cocalibrated $\G_2$-manifolds with Ricci flat characteristic connection}
%----------------------------------------------------------
%
% author and address
%   
%-------------------------------------------
%
\author{Thomas Friedrich}
%-------------------------------------------
\address{\hspace{-5mm} 
Thomas Friedrich\newline
Institut f\"ur Mathematik \newline
Humboldt-Universit\"at zu Berlin\newline
Sitz: WBC Adlershof\newline
D-10099 Berlin, Germany\newline
{\normalfont\ttfamily friedric@mathematik.hu-berlin.de}}
%
%-----------------------------------------------------------
\subjclass[2000]{Primary 53 C 25; Secondary 81 T 30}
%-----------------------------------------------------------
%-----------------------------------------------------------
\begin{abstract}
%---------------
Any $7$-dimensional cocalibrated $\G_2$-manifold admits a unique connection 
$\nabla$ with skew symmetric torsion (see  \cite{FriedrichIvanov}). We 
study these
manifolds under the additional condition that the $\nabla$-Ricci tensor 
vanishes.
In particular, we describe their geometry in case of a maximal number of 
$\nabla$-parallel vector fields. 
\end{abstract}
%-------------
\maketitle
%----------------
%\tableofcontents
%----------------
\pagestyle{headings}
%
%
%---------------------------------------------------------------------------- 
\section{Introduction}\noindent
%----------------------------------------------------------------------------
%
%
Consider a triple $(M^n \, , \, g \, , \, \T)$ consisting of a Riemannian
manifold $(M^n \, , \, 
g)$ equipped with a $3$-form $\T$. We denote by $\nabla^g$ , 
 $\Ric^g$ and $\Scal^g$ the Levi-Civita
connection, the Riemannian Ricci tensor and the
scalar curvature. The formula
\bdm
\nabla_X Y \ := \ \nabla^g_X Y \ + \ \frac{1}{2} \, \T(X \, , \, Y \, , \,
 - )
 \edm
defines a metric connection with torsion $\T$. We will denote by 
$\Ric^{\nabla}$ and $\Scal^{\nabla}$ its Ricci tensor and scalar
curvature
respectively. If the Ricci tensor $\Ric^{\nabla} = 0$ vanishes, then $\T$ is a
coclosed form, $\delta \T = 0$ ,  and the Riemannian Ricci tensor is completely given by the
$3$-form $\T$ (see \cite{FriedrichIvanov}), 
\bdm
\Ric^g(X \, , \, Y) \ = \ \frac{1}{4} \, \sum_{i,j=1}^n \T(X \, , \, e_i \, ,
\, e_j) \cdot \T(Y \, , \, e_i \, ,\, e_j) \ , 
\quad 
\Scal^g \ = \ \frac{3}{2} \, \|\T\|^2 \ .
\edm
In particular, the Ricci tensor $\Ric^g$ is non-negative, 
$\Ric(X , X) \geq 0$.\\

\noindent
Let us introduce the $4$-form $\sigma_{\T}$ depending on $\T$,
\bdm
\sigma_{\T} \ = \ \frac{1}{2} \, \sum_{i=1}^n \, (e_i \haken \T) \wedge (e_i
\haken \T) \ .
\edm
If moreover there exists a $\nabla$-parallel spinor field $\Psi$, then there
is an algebraic link between $d\T$, $\nabla \T$ and $\sigma_{\T}$
(see \cite{FriedrichIvanov}), 
\bdm
\big( X \, \haken \, d\T \ + \ 2 \, \nabla_X \T \big) \cdot \Psi \ = \ 0 \ ,
\quad \big( 3 \, d\T \ - \ 2 \sigma_{\T} \big) \cdot \Psi \ = \ 0 \ . 
\edm

\noindent
The classification of flat metric connections with skew symmetric torsion has
been investigated by Cartan and Schouten in 1926. Complete proofs are known
since the beginning of the 70-ties. In \cite{AgFr3} one finds a
simple proof of this result. Therefore, we are interested in
non-flat ($\mathcal{R}^{\nabla} \not\equiv 0$) and $\nabla$-Ricci flat ($\Ric^{\nabla}
\equiv 0$) metric connections with skew symmetric torsion $\T \not\equiv 0$.\\

\noindent
In this paper we study the $7$-dimensional case. Any cocalibrated 
$\G_2$-manifold admits 
a unique connection $\nabla$ with skew symmetric torsion and $\nabla$-parallel 
spinor field $\Psi$. 
If this characteristic connection is Ricci flat, then we obtain a 
solution of the 
Strominger equations (see \cite{FriedrichIvanov}), 
\bdm
\nabla \Psi \ = \ 0 , \quad \Ric^{\nabla} \ = \ 0 , \quad d*\T \ = \ 0 . 
\edm   
If $\T = 0$, $M^7$ is a Riemannian manifold with holonomy $\G_2$ and 
$\Ric^{g} =  0$ follows automatically.  
The case of $\T \not\equiv 0$ is different. The condition 
$\Ric^{\nabla} \equiv 0$ is not a consequence of the fact that the holonomy
of $\nabla$ is contained in $\G_2$, it is a new condition for the cocalibrated 
$\G_2$-structure.
In this paper we investigate the geometry of the $7$-manifolds under 
consideration. Moreover, we describe all these manifolds with a large
number of $\nabla$-parallel vector fields.
%
%---------------------------------------------------------------------------- 
\section{Examples of Ricci flat connections with skew symmetric torsion }\noindent
%----------------------------------------------------------------------------
%
Let us discuss some examples.
\begin{exa}
Any Hermitian manifold admits a unique metric connection $\nabla$ preserving
the complex structure and with skew symmetric torsion (see
\cite{FriedrichIvanov}) . In \cite{GGP} the authors constructed on
$(k-1) \, (S^2 \times S^4) \, \# \, k \, (S^3 \times S^3)$ a Hermitian
structure with vanishing $\nabla$-Ricci tensor, $\Ric^{\nabla} = 0$, for 
any $k \geq 1$ . These
examples are toric bundles over special K\"ahler $4$-manifolds.
\end{exa}
\begin{exa}
%----------
There are $7$-dimensional cocalibrated $\G_2$-manifolds $(M^7 \, , \, g \, ,
\, 
\omega^3)$ with characteristic torsion $\T$ such that
\bdm
\nabla \T \ = \ 0 \ , \quad d \T \ = \ 0 \ , \quad \delta \T \ = \ 0 \ , \quad 
\Ric^{\nabla} \ = \ 0 \ , \quad \hol(\nabla) \ \subset \
\un(2) \ \subset \ \g_2 \ . 
\edm
The regular $\G_2$-manifolds of this type have been described in \cite{Fri2},
Theorem 5.2  (the degenerate case   $ 2a+  c = 0$). $M^7$ is the
product $X^4 \times S^3$, where $X^4$ is a Ricci-flat K\"ahler manifold and
$S^3$ the round sphere. 
\end{exa}
\begin{exa}
A suitable deformation of any Sasaki-Einstein manifold yields a metric
connection with skew symmetric torsion and vanishing Ricci tensor, see 
\cite{AgFe}. 
\end{exa}

\noindent
Next we describe a similar method in order to construct $5$-dimensional
connections with skew symmetric torsion and vanishing Ricci tensor.
\begin{thm}\label{Theorem1}
%---------------------------
Let $(Z^4 \, , \, g \, , \, \Omega^2)$ be a $4$-dimensional Riemannian
manifold equipped with a $2$-form $\Omega^2$ such that
\begin{enumerate}
\item $d \Omega^2 \ = \ 0$ , $d * \Omega^2 \ = \ 0\ $ and $\ \Omega^2 \wedge
  \Omega^2 \ = \ 0$.
\item The $2$-dimensional distributions
\bdm
E^2 \ = \ \big\{ X \in TZ^4 \, : \ X \haken \Omega^2 \ = \ 0 \big\} \, , \quad
F^2 \ = \ \big\{ X \in TZ^4 \, : X \perp E^2 \big\}
\edm 
are integrable.
\item The $2$-form is of the form $\Omega^2 = 2 a \, f_1 \wedge f_2$, where
  $a$  is
  constant and $f_1, f_2$ is an oriented orthonormal frame in $F^2$.
\item The Riemannian Ricci tensor of $Z^4$ has two non-negative eigenvalues of
  multiplicity two, 
\bdm
\Ric^g \ = \ 4 a^2 \, \mathrm{Id} \ \mathrm{on} \ F^2  , \quad  
\Ric^g \ = \ 0 \ \mathrm{on} \ E^2  . 
\edm
\item $\Omega^2$ is the curvature form of some $\R^1$- or $S^1$-connection $\eta$.
\end{enumerate}
Then the principal fiber bundle $\pi : N^5 \rightarrow Z^4$
defined by $\Omega^2$ admits a Riemannian metric and the torsion form
\bdm
\T \ = \  \pi^*(\Omega^2) \wedge \eta
\edm
yields a metric connection $\nabla$  with the following properties:
\bdm
||\T||^2 \ = \ 4 a^2 \ , \ d\T \ = \ 0 \ , \ d*\T \ = \ 0 \ , \ \Ric^{\nabla} \ =
0 \ , \ \nabla \eta \ = \ 0 \ .
\edm
\end{thm}
\begin{proof}
%------------
Apply O'Neill's formulas and compute
\bdm
\Ric^g(X \, , \, Y) \ - \ \frac{1}{4} \, \sum_{i,j=1}^5 \T(X \, , \, e_i \, ,
\, e_j) \cdot \T(Y \, , \, e_i \, ,\, e_j) \ = \ 0 \ . \qedhere 
\edm
\end{proof}
\begin{exa}
Let $u = u(x,y)$ be a smooth function of two variables and consider the
metric
\bdm
g \ = \ e^u \, x \, \big(dx^2 \, + \, dy^2 \big) \ + \ x \, dz^2 \ + \ 
\frac{1}{x} \big( dt \, + \, y \, dz \big)^2 
\edm
defined on the set $Z^4 = \{(x,y,t,z) \in \R^4 : x > 0 \}$. $(Z^4 , g)$ is a
K\"ahler manifold and the Riemannian Ricci tensor has two eigenvalues, namely zero
and
\bdm
- \, \frac{u_{xx} + u_{yy}}{2 x e^u} ,
\edm
both  with multiplicity two (see  \cite{AAD} , \cite{LeBrun}). If the function
$u$ is a solution of the equation
\bdm
- \, \frac{u_{xx} + u_{yy}}{2 x e^u} \ = \ 4 a^2, 
\edm
Theorem \ref{Theorem1} is applicable and
we obtain a family of non-flat $5$-dimensional examples. Remark that a compact
K\"ahler manifold $Z^4$ of that type splits into $S^2 \times T^2$ , see \cite{ADM}. The corresponding
connection $\nabla$ on the Lie group $N^5 = S^3 \times T^2$ is flat, see
\cite{AgFr3} . 
\end{exa}
%

%---------------------------------------------------------------------------- 
\section{Cocalibrated  $\G_2$-manifolds with vanishing characteristic Ricci tensor }
%-----------------------------------------------------------------------
%
\noindent
Consider a cocalibrated $\G_2$-manifold $(M^7 \, , g \, , \, \omega^3)$, 
\bdm
d * \omega^3 \ = \ 0 , \quad \|\, \omega^3 \, \|^2 \ = \ 7 ,
\edm
and suppose that the $\G_2$-structure $\omega^3$ is not $\nabla^g$-parallel
(i.e. $d\, \omega^3 \not\equiv 0$).
There exists a unique metric connection $\nabla$ with skew symmetric
torsion and preserving the $\G_2$-structure $\omega^3$. Its torsion form is
given by the formula (see \cite{FriedrichIvanov}),
\bdm
\T \ = \ - \, * \, d \omega^3 \ + \ \mu \, \omega^3 \ , \quad
\mu \ = \ \frac{1}{6} \, \big( d \omega^3 \, , \, * \omega^3 \big) \ .
\edm 
The condition $\Ric^{\nabla} = 0$ becomes equivalent to $d \, \T = 0$ and $d * \T
= 0$. Indeed, we have:
\begin{thm}[{\cite[Thm 5.4]{FriedrichIvanov}}]\label{thm.equiv-cond}
%---------------------------------------------------------------------
The following conditions are equivalent:
\begin{enumerate}
\item $\Ric^{\nabla} \, = \, 0$.
\item $d\, \T \, = \, 0$ and $d * \T \, = \, 0$.
\item $d \mu \, = \, 0$ and $d * d \omega^3 \, - \, \mu \, d \, \omega^3 \, = \, 0$.
\end{enumerate}
\end{thm}
\noindent
Using the $\G_2$-splitting of $3$-forms, $\Lambda^ 3 = \Lambda^3_1 \oplus
\Lambda^3_7 \oplus \Lambda^3_{27}$, we know that  the characteristic torsion
of a cocalibrated $\G_2$-manifold belongs to $\T \in \Lambda^3_1 \oplus
\Lambda^3_{27}$. In particular, we obtain
\bdm
\T \, \wedge \, \omega^3 \ = \ 0 \ .
\edm
Differentiating the latter equation and using $d \, \T = 0$ one gets
\bdm
\big( * \, d \, \omega^3  -  \mu \, \omega^3) \wedge\omega^3 \ = \ 0 , 
\quad \|d \, \omega^3\|^2 \ = \ 6 \, \mu^2  .
\edm
We compute the length of $\T$,
\bdm
\| \T \|^2 \ = \ \| d \omega^3 \|^2  -  2 \, \mu  \big(* d \omega^3  , \,
\omega^3 \big) + 7 \, \|\omega^3\|^2 \ = \ 6 \, \mu^2 - 12 \, \mu^2  +
 7 \mu^2 \ = \ \mu^2 . 
\edm
Consequently, $\|\T \|^2$ is constant. Moreover, the Riemannian scalar curvature
is constant, too,
\bdm
\Scal^g \ = \ \frac{3}{2} \, \|\T \|^2 \ = \ \frac{3}{2} \, \mu^2 .
\edm
Since $( \T  , \omega^3) = \mu$, we
decompose the torsion form into two parts according to the splitting of
$3$-forms,
\bdm
\T \ = \ \T_1  +  \T_{27} , \quad \T_1 \ = \ \frac{1}{7} \, \mu \, \omega^3, 
\quad \T_{27} \ = \ -  *  d \omega^3 + \frac{6}{7} \mu \omega^3 .
\edm
\begin{cor}[{\cite[Remark 5.5]{FriedrichIvanov}}]
%-------------------------------------------------
Let $(M^7 \, , \, g \, , \, \omega^3)$ be a compact, cocalibrated
$\G_2$-manifold with $\Ric^{\nabla} = 0$ and $\T \not= 0$. Then the third
cohomology group is non-trivial,
\bdm
H^3( M^7 \, ; \, \R) \ \not= \ 0 \ .
\edm
\end{cor}
\begin{exa}
On the round sphere $S^7$ there exists a $\G_2$-structure (not cocalibrated)
such that $\mathcal{R}^{\nabla} = 0$ (see \cite{AgFr3}). 
In particular, the Ricci tensor vanishes,
$\Ric^{\nabla} = 0$. The characteristic torsion is coclosed, $\delta \T = 0$,
but not closed, $d\, \T \not= 0$.
\end{exa}
\begin{NB}
A cocalibrated $\G_2$-manifold with $\Ric^{\nabla} = 0$ and $\T \not\equiv 0$
cannot be of pure type $\Lambda^3_1$ or $\Lambda^3_{27}$. Indeed, if
\bdm
0 \ = \  \T_{27} \ = \ -  *  d \omega^3 + \frac{6}{7}\mu \omega^3 
\edm
we differentiate,
\bdm
0 \ = \ - \, d \, * \, d \, \omega^3 \ + \ \frac{6}{7} \, \mu \,
d \, \omega^3 
\edm
and combine the latter formula with equation (3) of Theorem
\ref{thm.equiv-cond}.  We conclude that $\mu = 0$, $d \omega^3 = 0$ and,
finally, $\T = 0$. The
second case, i.\,e.~$\T_1 = 0$, implies immediately $\mu = 0$ and $\T = 0$.
\end{NB}
\noindent
There exists a canonical $\nabla$-parallel spinor field $\Psi_0$ such that
\bdm
\nabla \Psi_0 \ = \ 0 \ , \quad \omega^3 \cdot \Psi_0 \ = \ - \, 7 \, \Psi_0 \ .
\edm
Since $\Lambda^3_{27} \cdot \Psi_0 =  0$ we obtain
\bdm
\T \cdot \Psi_0 \ = \ \T_1 \cdot \Psi_0 \ = \ - \, \mu \, \Psi_0 \ .
\edm
The integrability condition for a parallel spinor  (see \cite{FriedrichIvanov})
yields an algebraic
restriction for the derivative $\nabla \T$, namely
\bdm
\nabla_X \big( \T \cdot \Psi \big) \ = \ \big( \nabla_X \T \big) \cdot \Psi \
= \ 0 \ , \quad \sigma_{\T} \cdot \Psi \ = \ 0 \ , \quad \T^2 \cdot \Psi \ = \
\| \, \T \, \|^2 \, \Psi 
\edm
for any vector $X \in TM^7$ and any $\nabla$-parallel spinor field 
$\Psi$. In particular, the characteristic torsion $\T$ acts on the space of
all $\nabla$-parallel spinors. This condition is not so restrictive. For
example, the space of $3$-forms $\Sigma^3 \in \Lambda^3_{27}$ killing three
spinors has dimension $14$, the space killing four spinors has 
still dimension $9$.
%
%---------------------------------------------------------------------------- 
\section{$\nabla$-parallel vector fields}
%-----------------------------------------------------------------------
\noindent
Via the Riemannian metric we identify vectors with $1$-forms. Denote by
$\mathcal{P}^{\nabla}$ the space of all $\nabla$-parallel vector field 
($1$-forms). Any $\nabla$-parallel vector field $\theta$ is a Killing field and 
\bdm
2 \, \nabla^g \theta \ = \ d \, \theta \ = \ \theta \haken \T \ , \quad
\nabla^g _{\theta} \theta \ = \ 0 \ . 
\edm
holds. This formula together with $d \, \T = 0$ implies that $\T$ is preserved
by the flow of $\theta$,
\bdm
\mathcal{L}_{\theta} \T \ = \ 0 \ .
\edm
The Riemannian Ricci tensor on $\theta$  becomes
\bdm
\Ric^g( \theta \, , \, \theta) \ = \ \frac{1}{2} \, \|\, d \, \theta \, \|^2 \ .
\edm
The subgroup of $\G_2$ preserving four vectors in $\R^7$ is
trivial. The isotropy subgroups of two or three vectors in $\R^7$ coincide and
this group is isomorphic to $\SU(2) \subset \G_2$. Finally, the isotropy
subgroup of one vector is isomorphic to $\SU(3) \subset \G_2$ (see for example
\cite{Fri2}). This algebraic observation proves immediately the following
\begin{prop}
If $(M^7,g ,\omega^3)$ is not $\nabla$-flat, then the possible 
dimensions of  the space $\mathcal{P}^{\nabla}$ are $ 0 , \, 1$, or $3$. 
\end{prop}
%
%
%---------------------------------------------------------------------------- 
\subsection{The case of three $\nabla$-parallel vector fields}
%-----------------------------------------------------------------------
%
\noindent
We discuss the case that there are three orthonormal and 
$\nabla$-parallel $1$-forms $\theta_1
, \, \theta_2, \, \theta_3$. Then $\omega^3(\theta_1, \, \theta_2, \, -)$ is
$\nabla$-parallel, too. If it does not coincide with $\theta_3$ , then we have
at least four $\nabla$-parallel $1$-forms, i.e. the $\G_2$-connection $\nabla$
is flat. Under our assumption $\mathcal{R}^{\nabla} \not\equiv 0$ we conclude
that
\bdm
\omega^3(\theta_1 \, , \,  \theta_2 \, , \, - ) \ = \ \theta_3 \ , \quad
\omega^3(\theta_1 \, , \, \theta_2 \, , \, \theta_3 ) \ = \ 1 \ .
\edm
The holonomy of the connection $\nabla$ is contained in $\su(2) \subset \g_2$.
Moreover, the spinors
\bdm
\Psi_0 , \quad \Psi_1\ := \ \theta_1 \cdot \Psi_0  ,
\quad \Psi_2\ := \ \theta_2 \cdot \Psi_0 , 
\quad\Psi_3\ := \ \theta_3 \cdot \Psi_0 
\edm
are all $\nabla$-parallel spinors. The torsion form $\T$ acts as a symmetric
endomorphism on the space $\Lin(\Psi_0,\Psi_1,\Psi_2,\Psi_3)$ and
$\T \cdot \Psi_0 = - \, \mu \ \Psi_0$. Consequently, $\T$ acts on the
$3$-dimensional space $\Lin(\Psi_1,\Psi_2,\Psi_3)$ and $\T^2 = \| \T \|^2
\cdot \Id= \mu^2 \cdot \Id$. We decompose the torsion form into
\bdm
\T \ = \ \T_1 +  \T_{27} \ = \ \frac{1}{7} \, \mu \, \omega^3 + \T_{27}
\edm 
and we use the known action of $\omega^3$ on spinors:

\bdm
\omega^3 \cdot \Psi_0 \ = \ -  7 \Psi_0 , \quad
\omega^3 \cdot \Psi_i \ = \ \Psi_i,  \ i = 1,2,3 , \quad
\T_{27} \cdot \Psi_0 \ = \ 0 .
\edm
Finally, $\T_{27} \in
\Lambda^3_{27}$ preserves the space $\Lin(\Psi_1,\Psi_2,\Psi_3)$ and
\bdm
\T^2_{27} + \frac{2}{7} \, \mu \, \T_{27} \ = \ \frac{48}{49} \, \mu^2 .
\edm
Without loss of generality we may assume that $\Psi_1 , \, \Psi_2 , \, \Psi_3$
are eigenspinors of $\T_{27}$,
\bdm
\T_{27} \cdot \Psi_i \ = \ m_i \, \Psi_i \ , \quad m_i^2 \ + \
\frac{2}{7} \, m_i \, \mu \ = \ \frac{48}{49} \, \mu^2 , \ i \ = \ 1,2,3 .
\edm
We fix an
orthonormal basis $e_1 \, , \ldots \, , e_7$ such that
\bdm
\omega^3 \ = \ e_{127} \, + \, e_{135} \, - \, e_{146} \, - \, e_{236} \, - \,
 e_{245} \, + \, e_{347} \, + \, e_{567} 
\edm
and $\theta_1 = e_1, \, \theta_2 = e_2, \, \theta_3 = e_7$. This is possible,
since we already have $\omega^3(\theta_1, \theta_2, \theta_3) = 1$. Let
\bdm
\T_{27} \ = \ \sum_{i<j<k} t_{ijk} \, e_{ijk}
\edm 
be the $3$-form $\T_{27}$ and introduce the following numbers:
\bdm
a \ := \ t_{236} + t_{245} \, , \quad b \ := \ t_{347} + t_{567} ,
\quad c \ := \ t_{235} - t_{246} .
\edm
A purely algebraic computation yields the following
\begin{lem}
%----------
The space of all $3$-forms $\T_{27} \in \Lambda^3_{27}$ such that $\T_{27}
\cdot \Psi_i = m_i \, \Psi_i$, $i = 1,2,3$ is an affine space of dimension
$9$. A parameterization is given by
\begin{eqnarray*}
\T_{27} \ &=& \ \big( - \, \frac{m_1}{2} \, -  \, b \big) \, e_{127} \, - \,
t_{156} \, e_{134} \, + \, \big( \frac{m_1}{2} \, + \, t_{146} \, + a \big) \,
e_{135} \\
&&- \, t_{145}  \, e_{136} \, + \,  t_{145} \, e_{145} \, + \,  t_{146} \, e_{146} \, + \, t_{156}
\, e_{156} \, - \, t_{256} \, e_{234} \\
&&+ t_{235} \, e_{235} \, + \, t_{236} \, e_{236} \, + \, t_{245} \, e_{245} \,  
+ \,  t_{246} \, e_{246} \, + t_{256} \, e_{256} \, + \, t_{347} \, e_{347} \\
&&+ t_{467} \, e_{357} \, - \,  t_{457} \, e_{367} \, + \, t_{457} \, e_{457} \,
+ \, t_{467} \, e_{467} \, 
+ \, t_{567} \, e_{567} \ . 
\end{eqnarray*}
and
\bdm
m_1  + 2  a  +  2  b \ = \ m_2  , \quad
- 2  a + 2 b \ = \ m_3 , \quad c \ = \ 0 .
\edm
\end{lem}
\begin{cor}
For $X \perp \Lin(\theta_1, \theta_2, 
\theta_3)$ we have 
\bdm
\T( \theta_i \, , \, \theta_j \, , \, X) \ = \ 0 \ , \quad 
\T \ = \ ( \theta_1 \haken \T ) \wedge \theta_1 \ + \ 
( \theta_2 \haken \T ) \wedge \theta_2 \ + \  ( \theta_3 \haken \T ) \wedge
\theta_3 \ . 
\edm
\end{cor}

\noindent
We solve the linear system with respect to $a$ and $b$ :
\bdm
a \ = \ - \, \frac{1}{4} \, \big( m_1 \, - \, m_2 \, + \, m_3 \big) \ , \quad
b \ = \ \frac{1}{4} \, \big( - \, m_1 \, + \, m_2 \, + \, m_3 \big)
\edm
In particular, 
\bdm
 m_1 \, + \, 2 \, b \ = \ \frac{1}{2} \, \big( m_1 \, + \, m_2 \, + \, m_3
 \big) \ .
\edm
We are interested in the value 
\bdm
\T(\theta_1 , \theta_2  ,  \theta_3) \ = \, \frac{1}{7} \, \mu \,  
- \, \frac {m_1}{2} -  b \ = \ \frac{1}{7} \, \mu  -  \frac{1}{4}
\big( m_1  +  m_2  +  m_3 \big).
\edm
We have $8$ possibilities, namely 
\bdm
m_i \ = \ \frac{6}{7} \, \mu \ , \quad \mbox{or} \quad m_i \ = \ - \,
\frac{8}{7} \, \mu \ .
\edm
Therefore, 
\bdm
\T(\theta_1  , \theta_2  ,  \theta_3) \ = \ 0 , \ \pm \,
\frac{1}{2} \, \mu \quad \mbox{or} \quad \mu  .
\edm
We summarize the result.
\begin{thm} \label{Theorem3}
%---------------------------
Let $(M^7 ,g ,\omega^3)$ be a cocalibrated $\G_2$-manifold and $\nabla$
its characteristic connection. Suppose that $\Ric^{\nabla} = 0$, 
$\|\T\|^2 = \mu^2 > 0$  and
$\mathcal{R}^{\nabla} \not\equiv 0$. If $\theta_1 , \theta_2 , \theta_3$ are
three orthonormal and $\nabla$-parallel vector fields, then
\begin{enumerate}
\item $\omega^3(\theta_1 , \theta_2 , \theta_3) = 1$ .
\item $\T(\theta_1 , \theta_2 , \theta_3)$ is constant and has only four
  possible values, $ 0 , \, \pm \mu/2 , \, \mu$.
\item $\T(\theta_i , \theta_j , X ) = 0$ for $X \perp \Lin(\theta_1 , \theta_2
  , \theta_3)$.
\end{enumerate}
In particular
\begin{eqnarray*}
\T &=& ( \theta_1 \haken \T ) \wedge \theta_1+ 
( \theta_2 \haken \T ) \wedge \theta_2 + ( \theta_3 \haken \T ) \wedge
\theta_3 \\
&=& d \, \theta_1 \wedge \theta_1  +  
d \theta_2  \wedge \theta_2  +   d \theta_3 \wedge
\theta_3 \ . 
\end{eqnarray*}
and
\bdm
[ \theta_1 ,  \theta_2] \ = \ -  \T(\theta_1 , \theta_2 , \theta_3)\,\theta_3 
\edm
is proportional to $\theta_3$. The $3$-dimensional space $\Lin(\theta_1 ,
\theta_2 , \theta_3)$ is closed with respect to the Lie bracket and
is a Lie subalgebra of the Killing vector fields. This algebra is either
commutative or isomorphic to $\so(3)$.
\end{thm}
\begin{NB}
Since we do not assume that the torsion form $\T$ is $\nabla$-parallel, it is
not obvious by general arguments that $[\theta_1 \, , \, \theta_2 ] = - \,
\T(\theta_1 \, , \, \theta_2)$ is again $\nabla$-parallel. 
\end{NB}

\noindent
We can classify the case of $\T(\theta_1,\theta_2,\theta_3) = \mu$
immediately. Indeed, we have then $\|\T\|^2 \geq \mu^2$. On the other hand, we
know that $\|\T\|^2 = \mu^2$ holds. It follows that
\bdm
\T = \ \mu \, \theta_1 \wedge \theta_2 \wedge \theta_3  \quad \mbox{and}
\quad \nabla \T \ = \ 0 \ .
\edm   
Cocalibrated $\G_2$-structures with characteristic holonomy $\su(2)$ and a
characteristic torsion of the given type have been classified at the end of 
our paper \cite{Fri2}. We apply this result and obtain
\begin{thm}
%------------
Let $(M^7 ,g,\omega^3)$ be a complete, cocalibrated $\G_2$-manifold and $\nabla$
its characteristic connection. Suppose that $\Ric^{\nabla} = 0$. If 
$\theta_1 , \theta_2 , \theta_3$ are
three orthonormal and $\nabla$-parallel vector fields and $\T(\theta_1 ,
\theta_2 , \theta_3) = \mu$, then the universal covering of $M^7$ is isometric
to the product $X^4 \times S^3$, where $X^4$ is a complete anti-self dual and
Ricci flat Riemannian manifold.
\end{thm} 
\noindent
If $\T(\theta_1 , \theta_2 , \theta_3) = 0$ the $3$-dimensional
abelian Lie group acts on $M^7$ locally free as a group of isometries and 
preserves 
the torsion form $\T$. Moreover, we obtain the $2$-forms $d \theta_i = \theta_i 
\haken \T$ and
\bdm
\mathcal{L}_{\theta_i}(\theta_j \haken \T) \ = \ 0 \, , \quad
\theta_i \haken \theta_j \haken \T \ = \ 0 \ .
\edm 
We will investigate the special case , where two of these $2$-forms 
vanish, later.
\begin{NB}
We do not have any results in case of $|\T(\theta_1 , \theta_2 , \theta_3)| 
= \mu/2$. 
\end{NB}
%
%---------------------------------------------------------------------------- 
\subsection{Special $\nabla$-parallel vector fields}
%-----------------------------------------------------------------------
\noindent
There are special $\nabla$-parallel vector fields ($1$-forms), namely
\bdm
\mathcal{SP}^{\nabla} \ := \ \big\{ \theta \ : \ \nabla^g \theta \ = \ 0 \
\mbox{and} \ \theta \haken \T \ = \ 0 \big\} \subset\mathcal{P}^{\nabla} . 
\edm
A consequence of the formula in Theorem \ref{Theorem3} is the following
\begin{cor}
%-----------
If $\T \not\equiv 0$ and $\mathcal{R}^{\nabla} \not\equiv 0$, then 
$\dim (\mathcal{SP}^{\nabla}) \leq 2$.
\end{cor}
\begin{prop}
%-----------
If $\theta \in \mathcal{SP}^{\nabla}$ is special $\nabla$-parallel, then
\bdm
\nabla^g_{\theta} \, \omega^3 \ = \ 0  , \quad  d ( \theta \haken \omega^3) 
\ = \ 
\theta \haken d \, \omega^3 , \quad \mathcal{L}_{\theta} ( \theta \haken
\omega^3) \ = \ 0. 
\edm
\end{prop}
\begin{proof}
%------------
Since $\theta \haken \T = 0$ we get
\bdm
\nabla_{\theta} S \ = \ \nabla^g_{\theta} S \ + \ \frac{1}{2}\, \rho_*(\theta
\haken \T)(S) \ = \ \nabla^g_{\theta} S 
\edm
for any tensor S. Here $\rho_*$ denotes action
of $\so(7)$ in the corresponding tensor representation. In particular, 
\bdm
\nabla^g_{\theta} \, \omega^3 \ = \ 0 .
\edm
Since $\theta$ is $\nabla^g$-parallel, we have $\nabla^g(\theta \haken
\omega^3) = \theta \haken \nabla^g \omega^3$ . Using an orthonormal frame with
$\theta = e_7$ we compute the differential
\begin{eqnarray*}
d \, (\theta \haken \omega^3) &=& \sum_{i=1}^7 \nabla^g_{e_i}(\theta \haken
\omega^3) \wedge e_i \ = \ \sum_{i=1}^6 (\theta \haken \nabla^g_{e_i} \omega^3)
\wedge e_i \ + \ 0 \ = \ \sum_{i=1}^6 \theta \haken (\nabla^g_{e_i} \omega^3
\wedge e_i)  \\
&=&   \sum_{i=1}^6 \theta \haken (\nabla^g_{e_i} \omega^3 \wedge e_i) \ + \ 
\theta \haken (\nabla^g_{\theta} \omega^3 \wedge \theta) \ = \ \theta \haken d
 \omega^3 . 
\end{eqnarray*}
Finally, $\mathcal{L}_{\theta}( \theta \haken \omega^3) = \theta \haken
d(\theta \haken \omega^3) = \theta \haken \theta \haken \ d \, \omega^3 = 0$.
\end{proof}
\begin{thm}
%---------
Let $(M^7,g,\omega^3)$ be a compact, cocalibrated $\G_2$-manifold and $\nabla$
its characteristic connection. Suppose that $\Ric^{\nabla} = 0$, $\|\T||^2 =
\mu^2 > 0$  and
$\mathcal{R}^{\nabla} \not\equiv 0$. Then the space of harmonic $1$- forms
coincides with $\mathcal{SP}^{\nabla}$, 
\bdm
H^1 ( M^7 \, ; \, \R) \ = \ \big\{ \theta \, : \ \Delta^g \theta \ = \ 0
\big\} \ = \ \mathcal{SP}^{\nabla} \ .
\edm
In particular, the second Betti number is bounded, $b_2(M^7) \leq 2$.
\end{thm}
\begin{proof}
%-----------
The result follows directly from the Weitzenboeck formula for $1$-forms and the link between $\Ric^g$ and the torsion form $\T$, 
\bdm
0 \ = \ \int_{M^7} g (\Delta^g \theta \, , \, \theta ) \ = 
\int_{M^7} ||\nabla^g \theta||^2 + \int_{M^7} \Ric^g(\theta,\theta) \ =  
\int_{M^7} ||\nabla^g \theta||^2 +  \frac{1}{2} \, \int_{M^7}
\|\theta \haken \T\|^2  . 
\edm
\end{proof}
%
%---------------------------------------------------------------------------- 
\subsection{The case of two special $\nabla$-parallel vector fields}
%-----------------------------------------------------------------------
%
\noindent
Suppose that there exist two special $\nabla$-parallel vector fields
$\theta_1, \theta_2$,
\bdm
\nabla^g \ \theta_1 \ = \ \nabla^g \theta_2 \ = \ 0 \ , \quad 
\theta_1 \haken \T \ = \ \theta_2 \haken \T \ = \ 0 \ .
\edm
Then $\omega^3(\theta_1 , \theta_2, \, -) = \theta_3$ is the third
$\nabla$-parallel (non-special) vector field and we have
\bdm
\T(\theta_1 , \theta_2 ,\theta_3) \ = \ 0 , 
\quad [\theta_1 ,\theta_2 ] \ = \ [\theta_1 ,\theta_3] \ = \ 
[\theta_2,\theta_3] \ = \ 0 . 
\edm
The conditions $\theta_1 \haken \T = \theta_2 \haken \T = 0$ restrict the
algebraic type of the torsion form. In fact, Theorem \ref{Theorem3} yields that
the possible torsion forms depend on two parameters only. Indeed, there are
two possibilities.
The first case:
\bdm
a \ = \ \frac{2}{7} \, \mu , \ b \ = \ \frac{5}{7} \, \mu, \ 
m_1 \ = \ - \, \frac{8}{7} \, \mu , \ m_2 \ = \ m_3 \ = \ \frac{6}{7} \, \mu .
\edm
The second case: 
\bdm
a \ = \ \frac{2}{7} \, \mu , \ b \ = \ - \, \frac{2}{7} \, \mu , \ 
m_1 \ = \ \frac{6}{7} \, \mu , \ m_2 \ = \ \frac{6}{7} \, \mu , \
m_3 \ = \ - \, \frac{8}{7} \, \mu .
\edm
Introducing a new notation for the frame
\bdm
f_1 \ := \ e_3 , \ f_2 \ := \ e_4 , \ f_3 \ := \ e_5 , \ f_4 \ := \ e_6 , 
\ f_5 \ := \ e_7
\edm
we obtain the following formula for the torsion form: 
\begin{eqnarray*}
\T &=& (t_{125}  +  \mu/7)  f_{125}  +  t_{245}  (f_{135}  +  f_{245})  +  
t_{235}  ( -  f_{145}  +  f_{235})  +  (t_{345} +  \mu/7)  f_{345} \ , \\
b &=& t_{125} \, + \, t_{345} \ = \ \frac{5}{7} \, \mu \ \ \mbox{or} \ \ - \,
\frac{2}{7} \, \mu \\
\mu^2 &=& \|\T\|^2 \ = \ (t_{125} \, + \, \frac{\mu}{7})^2 \, + \, (t_{345}\,
+ \, \frac{\mu}{7})^2 \, + \, 2 \, t^2_{245}
\, + \, 2 \, t^2_{235}  .   
\end{eqnarray*}
If $M^7$ is complete, its universal covering splits into $N^5 \times
\R^2$ and the torsion $\T$ as well as the form $\theta_3 = e_7 = f_5$ are
forms on $N^5$. This follows form $\mathcal{L}_{\theta_i} \T = 0 \, , \, 
\mathcal{L}_{\theta_i} \theta_3 = 0$ for $i = 1,2$.  We reduced the dimension.
$(N^5 , g , \nabla , \T ,\theta_3)$ is a $5$-dimensional
Riemannian manifold equipped with a torsion form $\T$ as well as a metric
connection $\nabla$ such that
\bdm
d  *  \T \ = \ 0 \ , \, d \, \T \ = \ 0 , \ ||\T||^2 \ = \ 0 , \ 
\Ric^{\nabla} \ = \ 0 , \ \mathcal{R}^{\nabla} \ \not\equiv \ 0 , \
\hol(\nabla) \subset \ \su(2) \ \subset \ \g_2  
\edm
hold. $\theta_3$ is  $\nabla$-parallel on $N^5$, 
\bdm
\nabla \theta_3 \ = \ 0 , \ d \theta_3 \ = \ \theta_3 \haken \T  , 
\ \T \ =
\ \theta_3 \wedge d \, \theta_3 , \ 0 \ = \ d \, \T \ 
= \ d \theta_3\wedge d  \theta_3 .
\edm
Consider the case of $b = - 2 \mu /7$. Then 
\bdm
t_{125}  +  \frac{\mu}{7} \ = \ -  t_{345} \, - \, \frac{\mu}{7}
\edm
and we obtain
\bdm
* \, \T = \ - \, \theta_3 \haken \T \ = \ -  d  \theta_3 , \quad 
*  d \theta_3 \ = \ -  \T \ = \ - \, d \theta_3 \wedge \theta_3 . 
\edm
We multiply the latter equation by $d \theta_3$:
\bdm
\| d \theta_3\|^2 \ = \ d \theta_3 \wedge *  d\theta_3 \ = \ - 
\theta_3\wedge d \theta_3 \wedge d \theta_3 \ = \ 0 .
\edm
Consequently, $b = -2 \, \mu/7$ implies that the torsion form vanishes, $\T =
0$, i.e. the second case is impossible.\\

\noindent
We observe that there
are three $\nabla$-parallel $2$-forms on $N^5$, namely, 
\bdm
\Omega^2_i \ := \ \theta_i \haken \big(\omega^3  -  \theta_1  \wedge \theta_2 \, \wedge
\, \theta_3) .
\edm
Consequently, $\hol(\nabla) \subset \ \su(2)$. We can express these forms
in our local frame,
\begin{eqnarray*}
\Omega^2_1 &=&  f_{13} \, - \, f_{24} , \quad 
\Omega^2_2 \ = \ - \, f_{14} \, - \, f_{23} , \quad 
\Omega^2_3 \ = \ \ f_{12} \, + \, f_{34} . 
\end{eqnarray*}
Remark that
\bdm
\big( \theta_3 \haken \T \, , \, \Omega^2_1 \big) \ = \ 
\big( \theta_3 \haken \T \, , \, \Omega^2_2 \big) \ = \ 0 \ , \quad
\big( \theta_3 \haken \T \, , \, \Omega^2_3 \big) \ = \  b \, + \, \frac{2}{7} 
\, \mu \ = \ \mu 
\edm
holds.
\begin{thm}
%----------
The kernel of $\T$
\bdm
E^2 \ := \ \big\{ X \in TN^5 \, : \ X \haken \T \ = \ 0 \big\} 
\edm
is a $2$-dimensional subbundle of $TN^5$. The tangent bundle splits into
two subbundles of dimension $2$ and $3$, respectively,
\bdm
TN^5 \ = \ E^2 \ \oplus \ (E^2)^{\perp} .
\edm
$\theta_3$ belongs to $(E^2)^{\perp}$ and 
the torsion form is given by
\bdm
\T \ = \ \mu \, f_1^* \wedge f_2^* \wedge \theta_3 ,
\edm
where $f_1^* , f_2^* , \theta_3$ is an orthonormal basis in $(E^2)^{\perp}$. 
Both subbundles are involutive and $N^5$ splits locally (but the $2$- und
$3$-dimensional leaves are not totally geodesic).
\end{thm}
\begin{proof}
%------------
We compute the determinant of the skew symmetric endomorphism $\theta_3
\haken \T$ on the space of all vectors being orthogonal to $\theta_3$,
\bdm
\mbox{Det}(\theta_3 \haken \T) \ = \ \frac{1}{4} \big( - \, b^2 \, - \, \frac{4}{7} \, b
\, \mu \, + \, \frac{45}{49} \, \mu^2 \big)^2 \ = \ 0 \ . 
\edm
This proves that the dimension of $E^2$ equals two. Let $f^*_1 , f^*_2 , 
f^*_3 , f^*_4 , f^*_5 = \theta_3$ be an orthonormal frame such that
\bdm
\Lin(f^*_1 \, , \, f^*_2 \, , \, f^*_5 ) \ = \ \big(E^2 \big)^{\perp} \ ,
\quad
\Lin(f^*_3 \, , \, f^*_4 ) \ = \ E^2  \ . 
\edm
Since $\mu$ is constant and $d \, \T \, = \, d\, * \T \, = \, 0$ we have
\bdm
d \big( f^*_1 \wedge f^*_2 \wedge f^*_5 \big) \ = \ 0 \ , \quad
d \big( f^*_3 \wedge f^*_4 \big) \ = \ 0 \ . 
\edm
We differentiate the equations $f^*_3 \wedge f^*_3 \wedge f^*_4 = 0 \, , \,
f^*_4 \wedge f^*_3 \wedge f^*_4 = 0$,
\begin{eqnarray*}
0 &=& d  f^*_3 \wedge (f^*_3 \wedge f^*_4) \, - \, f^*_3 \wedge d 
(f^*_3 \wedge f^*_4) \ = \  d  f^*_3 \wedge (f^*_3 \wedge f^*_4) \\
0 &=& d  f^*_4 \wedge (f^*_3 \wedge f^*_4) \, - \, f^*_4 \wedge d 
(f^*_3 \wedge f^*_4) \ = \  d  f^*_4 \wedge (f^*_3 \wedge f^*_4) \ .
\end{eqnarray*}
By the Frobenius Theorem, the bundle $(E^2)^{\perp}$ is involutive. Similarly
we have
\bdm
df^*_1 \wedge (f^*_1 \wedge f^*_2 \wedge f^*_5) \ = \ 
df^*_2 \wedge (f^*_1 \wedge f^*_2 \wedge f^*_5) \ = \ 
df^*_5 \wedge (f^*_1 \wedge f^*_2 \wedge f^*_5) \ = \ 0 
\edm
and the bundle $E^2$ is involutive.
\end{proof}
\noindent
This splitting is not $\nabla$-parallel ($\nabla \T \not= 0)$, but 
the flow of $\theta_3$ preserves the splitting ($\mathcal{L}_{\theta_{3}} \T = 0$). The Ricci tensor preserves the
splitting, too. Indeed, it depends only on $\T$ and we compute easily:
\begin{thm}
%----------
The Ricci tensor $\Ric^g$ preserves the splitting of the tangent bundle and
\bdm
\Ric^g_{|E^2} \ = \ 0 , \quad \Ric^g_{|(E^2)^{\perp}} \ = \ \frac{1}{2} 
\mu^2 \, \Id . 
\edm
In particular, the Ricci tensor of $(N^5 , g)$ has  constant
eigenvalues, and these are $0$ and $\mu^2/2 > 0$.
\end{thm}
\noindent
The $2$-form $d \theta_3$ is invariant under the flow of $\theta_3$, 
\bdm
\mathcal{L}_{\theta_3} \big(d \theta_{3}\big) \ = \ 0  \quad \mbox{and} \quad d
\theta_3 \wedge d \theta_3 \ = \ 0  . 
\edm
If the orbit space $Z^4 := N^5/\theta_3$ is smooth, its tangent bundle
splits into two involutive $2$-dimensional subbundles. $d \theta_3$ defines
a $2$-form on $Z^4$ satisfying all the conditions of Theorem \ref{Theorem1}.
However, we have an additional condition for $(N^5, g, \nabla, \T,
\theta_3)$, namely the holonomy of $\nabla$ should be contained in $\su(2)
\subset \g_2$ and the holonomy representation is in $\C^2 \subset \R^5$. This
is equivalent to the condition that there are three $\nabla$-parallel
$2$-forms
$\Omega^2_1 , \Omega^2_2 ,\Omega^2_3$. The $2$-form $\Omega^2_3$ plays a
special role on $N^5$. Indeed , it projects down to a K\"ahler form on $Z^4$.
\begin{prop}
%-----------
\bdm
\nabla \, \Omega^2_3 \ = \ 0 , \quad d \, \Omega^2_3 \ = \ 0 , \quad
\mathcal{L}_{\theta_3} \Omega^2_3 \ = \ 0  .
\edm
In particular, if $Z^4$ is smooth, then $\Omega^2_3 \in \Lambda^2_{+}(Z^4)$
defines a $\nabla^g$-parallel, self-dual $2$-form on $Z^4$.
\end{prop}
\begin{proof}
%--------------
Using the frame $f_1 , \ldots , f_5$ one easily computes the formula
\bdm
 \Omega^2_3 \ = \ \frac{1}{\mu} \, \big( * \T \, + \, d \, \theta_3 \big)
\ = \ \frac{1}{\mu} \, \big( * \T \, + \, \theta_3 \haken \T \big) \ .
\edm
Since $d * \T = 0$ we obtain $ d \, \Omega^2_3 = 0$. Moreover,
$\mathcal{L}_{\theta_3} \T = 0$,  and 
\bdm
\mathcal{L}_{\theta_3} \Omega^2_3 \ = \ \frac{1}{\mu} \,
\mathcal{L}_{\theta_3} (d \theta_3) \  = \ \frac{1}{\mu} \big( \theta_3 \haken
(\theta_3 \haken \T) \big) \ = \ 0 \ . \qedhere
\edm
\end{proof}
\noindent
A similar algebraic computation yields the following formulas.
\begin{prop}
%------------
\begin{eqnarray*}
d \, \Omega^2_1 &=& \mu \, \Omega^2_2 \wedge \theta_3 \ , \quad
d \, \Omega^2_2 \ = - \, \mu \, \Omega^2_1 \wedge \theta_3 \ , \\
\mathcal{L}_{\theta_3} \Omega^2_1 &=& \mu \, \Omega^2_2 \ , \quad \quad \
\mathcal{L}_{\theta_3} \Omega^2_2 \ = \ - \, \mu \, \Omega^2_1 \ .
\end{eqnarray*}
\end{prop}
\begin{proof}
%------------
Since the $2$-forms are $\nabla$-parallel, we can compute the derivatives
using the formula (see \cite{AgFr1})
\bdm
d\, \Omega^2 \ = \ \sum_{j=1}^5 (f_j \haken \Omega^2) \wedge (f_j \haken
\T) \ . \qedhere
\edm
\end{proof}
\begin{NB}
%---------
In the frame $f^*_1, \ldots , f^*_5$ we have
$\Omega^2_3 = f^*_1 \wedge f^*_2 \, + \, f^*_3 \wedge f^*_4$ , too. In
particular, $\Omega^2_3$ is completely defined by $\T$ and $\theta_3$. If
$Z^4$ is smooth and compact, then $Z^4 = S^2 \times T^2$, see \cite{ADM}, 
and the
connection $\nabla$ on $M^7 = N^5 \times \R^2 = S^3 \times T^2 \times \R^2$ becomes flat.
\end{NB}
%

%---------------------
%---------------------
%----------------------------------
%\addcontentsline{toc}{section}{Literature}
%------------------------------------------
    
\end{document}